\documentclass[12pt]{article}
\usepackage{amssymb}
\usepackage{amsmath}
\usepackage{amsthm}
\usepackage{graphicx}

\usepackage{color}

\numberwithin{equation}{section}
\newcommand{\N}{{\mathbb N}}
\newcommand{\R}{{\mathbb R}}

\newcommand{\Rd}{{{\mathbb R}^{d}}}
\newcommand{\Sdmone}{{{\mathbb S}^{d-1}}}

\newcommand{\bbS}{{\mathbb S}}
\newcommand{\bbN}{{\mathbb N}}

 \def\t{{\theta}}
 \def\l{{\lambda}}
 \def\d{{\delta}}

 \def\CL{{\mathcal L}}
 \def\CO{{\mathcal O}}

 \def\NN{{\mathbb N}}

\newtheorem{thm}{Theorem}[section]

\newtheorem{lem}[thm]{Lemma}
\newtheorem{prop}[thm]{Proposition}

\newtheorem{conj}[thm]{Conjecture}

\theoremstyle{remark}
\newtheorem{rem}{Remark}[section]

\newcommand{\extras}[1]{}
\begin{document}
\pagenumbering{arabic}

\vspace{-5ex}
\title{A P\'olya criterion for  (strict) positive definiteness on the sphere}
\author{R.K. Beatson, W. zu Castell, Y. Xu\footnote{The work of this author was supported in part by NSF Grant DMS-1106113}}
\date{\today}
\maketitle

\begin{abstract}
Positive definite functions are very important in both theory and applications of
approximation theory, probability and statistics. In particular, identifying strictly positive definite kernels is of
great interest as interpolation problems corresponding to these kernels are guaranteed to be poised.
A  Bochner type result of Schoenberg  characterises continuous positive definite zonal functions, 
$f(\cos \cdot)$,  on the sphere $\Sdmone$, as those with nonnegative Gegenbauer coefficients.
More recent results characterise strictly positive definite functions on $\Sdmone$
 by stronger conditions on the signs of the Gegenbauer coefficients. Unfortunately, given a function $f$,
 checking the signs of all the Gegenbauer coefficients can be an onerous, or impossible,  task. Therefore, it is natural to seek simpler sufficient conditions which guarantee (strict) positive definiteness.  We state a
conjecture which leads to a P\'olya type criterion for functions to be (strictly) positive
definite on the sphere $\Sdmone$. In analogy
to the case of the Euclidean space, the conjecture claims
positivity of a certain integral involving Gegenbauer
polynomials. We provide a proof of the conjecture for $d$ from $3$ to $8$. 

\end{abstract}

\section{Introduction}
\label{sec:introduction}

Positive definite functions are very important in both theory and applications of
approximation theory, probability and statistics. Although  Bochner's theorem characterises
continuous positive
definite functions on $\Rd$ it has long been recognised that the conditions of Bochner's theorem may be difficult to check.
For example Askey~\cite{Askey_jmaa75} states
\begin{quotation}
``It is an unfortunate fact that necessary and sufficient conditions are often
impossible to verify and one must search for useful sufficient conditions when
confronted with a particular example.''
\end{quotation}
In 1918 P\'olya \cite{Po18} proved that an even, continuous
function which is convex on the positive real line and vanishes at infinity
has a non-negative Fourier transform. Later \cite{Po49} he proved a similar statement
for the inverse transform,  thereby showing that an even function $f$
which is continuous and convex on $[0,\infty)$, and vanishes at infinity, is
positive definite on the real line. This sufficient condition is now commonly referred
to as the \emph{P\'olya criterion}.

The criterion has been generalized for positive definite functions on
$\Rd$ \cite{As73,Ru58,Tr89}, considering radial functions instead of
even ones. Further refinements of these criteria followed (cf. \cite{Gn01}
and the references cited therein). Askey's proof \cite{As73} relates
his P\'olya type criterion to the non-negativity of certain integrals
\cite{Fi75,Ga75}. These integrals represent the Fourier transform of
a function, which can in some sense be considered as a prototype for
all functions Askey's P\'olya type criterion is applicable to.

A continuous function $g:[0,\pi] \to \R$ is (zonal) positive definite on the sphere
$\Sdmone$ if for all distinct point sets $X = \{x_1,\ldots, x_n\}$ on the sphere and
all $n\in\N$, the matrices  $M_X:=\left [ g(d(x_i,x_j))\right]_{i,j =1}^n$ 
are positive semi-definite, that is, $c^T M_X c \ge 0$ for all $c\in\R^n$. In this definition
$d(x,y)$ denotes the geodesic distance, $\arccos(x^T y)$, on $\Sdmone$. The function $g$ is 
\emph{strictly positive definite} 
on $\Sdmone$ if the matrices are all positive definite, 
that is, $c^T M_X c > 0$, for all nonzero $c \in \R^n$. The importance of strict positive definiteness 
is its connection with the poisedness of interpolation. Thus, if $g$ is strictly positive definite on $\Sdmone$
 then  there is exactly one function
of the form
$$
s(x) = \sum_{j=1}^n \lambda_j g( d(x,x_j)  ),
$$  
which takes given values, $\{ y_j\}_{j=1}^n$, at the distinct nodes
${\mathcal X} = \{x_j\}_{j=1}^n  \subset \Sdmone$. Therefore, an easy means of identifying strictly positive definite kernels is of
great interest as it will enable the assembly of a toolkit of different kernel based interpolation methods.

Continuous positive definite zonal functions on the sphere have been studied
by Schoenberg \cite{Sc42} who proved the following theorem of Bochner type
(Theorem 1 in \cite{Sc42}). 

\begin{thm} \label{th:schoenberg}
Let $f$ be a continuous function on $[-1,1]$. The function $f(\cos \cdot)$  is positive definite on
$\Sdmone$
if and only if $f(\cos\theta)$ has a Gegenbauer expansion
\begin{equation}
\label{eq:gegenbauer_expansion}
f(\cos\theta) \, = \,
\sum_{k=0}^\infty a_k C_k^{\frac{d-2}2}(\cos\theta),
\qquad \theta\in[0,\pi],
\end{equation}
in which all of
the coefficients $a_k$, $k\in\N_0$, are nonnegative, and
$ \sum_{k=0}^\infty a_k C_k^{\frac{d-2}{2}}(1) < \infty$.
\end{thm}

The characterization of strictly positive definite functions on $\Sdmone$ came somewhat later.
A simple sufficient condition~\cite{Xu92} states that $f(\cos \cdot)$ is strictly positive definite if, 
in addition to the conditions of Theorem~\ref{th:schoenberg},
all the Gegenbauer coefficients $a_k$ are positive.  
Chen, Menegatto and  Sun~\cite{Ch03} showed that
a necessary and sufficient condition for $f(\cos \cdot)$
be strictly positive definite on $\Sdmone$, $d \geq 3$, is that, in addition to the conditions of 
Theorem~\ref{th:schoenberg}, infinitely many of the Gegenbauer coefficients  with odd index, and infinitely many of those with even index,
are positive.
 
Schoenberg also characterized those
continuous functions $f:[-1,1]\to\R$ such that
$f(\cos \cdot)$ is a positive definite zonal function on all spheres
(cf. (1.6) and Theorem 2 in \cite{Sc42}).

\begin{thm}
\label{th:schoenberg_2}
Let $f$ be a continuous function on $[-1,1]$. The function $f(\cos \cdot)$ is positive definite on all spheres
$\Sdmone$, $d\ge 2$, if and only if $f(\cos\theta)$ has an expansion
\begin{equation}
\label{eq:abs_monotonic_expansion}
f(\cos\theta) \, = \, 
\sum_{k=0}^\infty a_k (\cos\theta)^k,
\qquad \theta\in[0,\pi],
\end{equation}
where $a_k\ge 0$ for all $k\in\N_0$ and $\sum_{k=0}^\infty a_k$ converges.
\end{thm}

A function $f:(c,d)\to \R$ is {\em absolutely monotonic}, if
it has derivatives of all orders on $(c,d)$, and these are all nonnegative. Such a function is  characterized by having a series expansion
\begin{equation} \label{eq:absolute_monotonicity}
f(x) = \sum_{k=0}^\infty f^{(k)}(c^+) (x-c)^k,
\end{equation}
converging to $f(x)$ for all $x\in (c,d)$.
If in addition $f$ is continuous on $[c,d]$ then the expansion
converges to $f$ uniformly on $[c,d]$. Therefore, Theorem~\ref{th:schoenberg_2} identifies
the continuous functions $f:[-1,1] \to \R$  such that  $f(\cos \cdot )$ is positive definite on all spheres, as those for which $f$ is
the analytic extension to $[-1,1]$ of an absolutely monotonic function
on $[0,1]$.

Note that a similar result holds true for radial functions in $\Rd$, too, with 
absolute monotonicity being replaced by complete monotonicity.
Theorem \ref{th:schoenberg_2} hints towards what should be a P\'olya type 
criterion for positive definite, zonal functions on the sphere. If we assume 
$f(\cos \cdot)$ to be positive definite on $\Sdmone$ only, we can expect the pattern
(\ref{eq:absolute_monotonicity}) to break down after finitely many terms.
This intuition fully applies in the Euclidean case $\Rd$.

The purpose of the present paper is to formulate a P\'olya type criterion for 
positive definite, zonal functions on the sphere. The proof actually depends
on proving non-negativity of a certain integral, see Conjecture \ref{conj}
below. In parallel to the Euclidean theory, the proof in the general case 
turns out to be quite hard. We were able to establish the result for dimensions
up to $d=8$, doing extensive computer algebra and  numerical calculations for higher dimensions.

Before stating the conjecture, let us first formulate the P\'olya type criterion.

\begin{thm}\label{th:polya_criterion}
Let $d\in\{3,4,\ldots,8\}$ and $\lambda=\lceil \frac{d-2}{2}\rceil$. 
Let the real-valued function
$g(\cdot)=f(\cos\cdot)$ on $[0,\pi]$ satisfy the following conditions:
\begin{itemize}
\item[(i)] $g\in C^{\lambda}[0,\pi]$, 
\item[(ii)] $\text{supp}(g) \subset [0,\pi)$,
\item[(iii)] the derivative, from the right, $g^{(\lambda+1)}(0)$ exists, and is finite,
\item[(iv)] $(-1)^{\lambda} g^{(\lambda)}$ is convex.
\end{itemize}
Then $g$  is a positive definite function on $\Sdmone$.

\medskip

If,  in addition to the above properties,  $g^{(\lambda)}$, restricted to $(0,\pi)$, does not reduce to a linear
polynomial,  then $g$ is a  strictly positive definite function on $\Sdmone$.

\end{thm}

We conjecture Theorem~\ref{th:polya_criterion} to be true for all dimensions $d>2$. Therefore, we
will provide a proof for general dimension, relying on the following 
conjecture.

\begin{conj}\label{conj}
Let $\delta>0$, $\lambda > 0$ and $n\in\N_0$. 
For every $0<t<\pi$, define
\begin{equation}
\label{eq:define_F}
F_n^{\lambda,\delta}(t) 
\, = \,
\int_0^t (t-\theta)^\delta \, C_n^\lambda(\cos\theta)
(\sin\theta)^{2\lambda}\, d\theta.
\end{equation}
Then $F_n^{\lambda,\delta}(t)> 0$  for all $t$ in $(0,\pi]$
if and only if $\delta\geq\lambda+1$.
\end{conj}

The conjecture is essentially equivalent to proving
$(t-\theta)^\delta_{+}$ is strictly positive definite on $\Sdmone$.
These zonal functions are supported on a spherical cap, as are the functions shown to be 
(strictly) positive definite in Theorem~\ref{th:polya_criterion}.

\medskip
The conjecture has been stated in greater generality than we really need for the 
positive definiteness results
for $\Sdmone$ of the form of Theorem~\ref{th:polya_criterion}.
For those we are only interested in Gegenbauer coefficients for expansions with integer parameter, 
$\lambda =\lceil (d-2)/2 \rceil$. 

\smallskip
The boundary case $\delta = \lambda+1$ of Conjecture~\ref{conj} 
will get special attention.
Therefore, we define
\begin{equation} \label{defn_Fsuplambda}
F_n^\lambda(t) =
 F_n^{\lambda,\lambda+1}(t) = \int_0^t (t-\theta)^{\lambda+1} C_n^\lambda(\cos \theta) (\sin \theta)^{2\lambda} \, d\theta .
\end{equation}
\smallskip
Those cases in which we have proven the conjecture are listed in the following proposition.

\begin{prop}\label{propn:proven_cases}
Let  $d \in \{  4, 6, 8\}$, $\lambda= (d-2)/2$ and $n\in \N_0$. Then 
$$
F_n^{\lambda}(t) 
\, = \,
\int_0^t (t-\theta)^{\lambda+1}\, C_n^\lambda(\cos\theta)
(\sin\theta)^{2\lambda}\, d\theta > 0,
$$
for all $0 < t \leq \pi$.
\end{prop}

\bigskip
The observant reader will have noticed that the case $d=2$,  that is the case of the circle $\bbS^1$,
 does not appear in 
Theorem~\ref{th:polya_criterion}. This is because this particular case does  not fit the general pattern for strict positive
definiteness.
Regarding positive definiteness Gneiting~\cite{Gn98} has shown
\begin{thm}
Suppose the function $g(t)$, defined for $t\in  [-K,K]$, has the following properties:
\begin{itemize}
\item[(i)] $g(t)$ is real--valued, even, and continuous,
\item[(ii)] $g(0)=1$,
\item[(iii)] $ \displaystyle \int_{-K}^K g(t) \, dt \geq 0$,
\item[(iv)] $g(t)$ is nonincreasing and convex for $t \in [0,K]$.
\end{itemize}
Then $g(t)$, $t \in[-K,K]$, is a correlation function on the circle circumference $2K$.
\end{thm}

In the case $d=2$, that is $\lambda=0$, the function  $F^0_n(t)$ is a multiple of $1-\cos(nt)$. Therefore, it is not positive on all
points $t \in (0,\pi]$, but rather has zeros at the points $t=2k\pi/n $, $0 < k \leq \lfloor n/2 \rfloor$. Consequently, in order to guarantee all the Gegenbauer coefficients of $g$ are strictly positive, as part of showing $g$ is strictly positive definite,  we have to assume slightly more than was required when $\lambda \in \bbN$. The methods used to show Theorem~\ref{th:polya_criterion}, with obvious modifications, yield the following:

\begin{thm}\label{th:polya_criterion_S1}
Let the real-valued function
$g(\cdot)=f(\cos\cdot)$ on $[0,\pi]$ satisfy the following conditions:
\begin{itemize}
\item[(i)] $g\in C[0,\pi]$, 
\item[(ii)] $\text{supp}(g) \subset [0,\pi)$,
\item[(iii)] the derivative, from the right, $g'(0)$ exists, and is finite,
\item[(iv)] $g$ is convex.
\end{itemize}
Then $g$  is a positive definite function on $\bbS^1$.

\medskip

If,  in addition to the above properties,  $g$, restricted to $(0,\pi)$, does not reduce to a piecewise linear
function with finitely many pieces,  then $g$ is a  strictly positive definite function on $\bbS^1$.
\end{thm}

We omit the proof.

\bigskip
{\em Notation:} In the body of the paper many expressions will   occur with removable singularities, consider for example 
(\ref{connection2}) with
$\lambda - \mu$ a negative integer. We interpret such expressions in the usual way, as
the  value of the limit.

\section{A Polya criteria for $\bbS^{d-1}$}

In this section we give a proof of  Theorem~\ref{th:polya_criterion} assuming the results regarding the positivity of the Gegenbauer coefficients of the functions $(t-\theta)^\mu_{+}$, $\mu=\lceil \lambda+1 \rceil$, listed in  Proposition~\ref{propn:proven_cases}. 

\medskip
Recall that the Gegenbauer expansion of a function $g = f(\cos \cdot)$ is
$ g \sim \sum a_n C^\lambda_n$, where $a_n = b_n/h_n$,
\begin{equation}
b_n = b_n(g) =b_n^{\lambda}(g) = \int_{-1}^1 f(x) C_n^{\lambda}(x) w_\lambda(x) \, dx 
       = \int_{0}^\pi f(\cos \theta)\, C_n^\lambda (\cos \theta) \left( \sin \theta \right)^{d-2} 
\,  d\theta, \label{eq:bn}
\end{equation}
and
\begin{equation} \label{eq:hn}
h_n = h^\lambda_n = \int_{-1}^1 \left\{C^\lambda_n(x) \right\}^2 w_\lambda(x) \,  dx > 0. 
\end{equation}

We first give an alternative expression for the coefficients, $\{b_n\}$, which will provide 
a major part of a proof of
the conjecture when the function $g$ has two more continuous derivatives than is assumed in  Theorem~\ref{th:polya_criterion}.
\begin{lem} \label{lem:alt_exp_Gegenbauer_coeffs}
If $g\in C^{\lambda+2}[0, \pi)$ is identically zero in a neighbourhood of $\pi$
then the coefficients, $\{b_n\}$, defined in (\ref{eq:bn}), are
alternatively given by the expression
\begin{equation}
b_n (g) =  \frac{(-1)^{\lambda+2}}{(\lambda+1)!}\,\int_0^\pi 
F_n^\lambda (\tau)  \; g^{(\lambda+2)}(\tau)\, d\tau.    \label{eq:geg_coeff_in_terms_of_F} 
\end{equation}
\end{lem}
{\em Proof:}
Applying Taylor's theorem with integral remainder
\begin{align*}
g(\theta)  & =  f(\cos \theta)\\
& =\sum_{k=0}^{\lambda+1} \frac{g^{(k)}(\pi)}{k!} (\theta-\pi)^k
\, + \,
\frac {1}{(\lambda+1)!}\,\int_\pi^\theta g^{(\lambda+2)}(\tau)(\theta-\tau)^{\lambda+1} \, d\tau \\
& =
\sum_{k=0}^{\lambda+1} \frac{(-1)^k g^{(k)}(\pi)}{k!} (\pi-\theta)^k
\, + \,
\frac{(-1)^{\lambda+2}}{(\lambda+1)!}\,\int_\theta^\pi g^{(\lambda+2)}(\tau)(\tau-\theta)^{\lambda+1}\, d\tau,
\quad \theta\in [0,\pi].
\end{align*}
Therefore,
\begin{align*}
b_n & =  \sum_{k=0}^{\lambda+1} \frac{(-1)^k g^{(k)}(\pi)}{k!} \int_0^\pi (\pi-\theta)^k \,
 C_n^\lambda(\cos \theta) \left( \sin \theta \right)^{2\lambda} \, d\theta
\nonumber \\
&\mbox{} \hspace{3ex} \ + \, \frac{(-1)^{\lambda+2}}{(\lambda+1)!}\,\int_0^\pi \int_\theta^\pi g^{(\lambda+2)}(\tau)(\tau-\theta)^{\lambda+1}\, d\tau \;  \; C_n^\lambda(\cos \theta) \left( \sin \theta \right)^{2\lambda} \, d\theta
\end{align*}
Since all the derivatives at $\pi$ vanish,
\begin{align*}
b_n & = \frac{(-1)^{\lambda+2}}{(\lambda+1)!}\,\int_0^\pi \int_0^\tau
(\tau-\theta)^{\lambda+1}\, C_n^\lambda(\cos \theta) \left( \sin \theta \right)^{2\lambda}  d\theta \; \; g^{(\lambda+2)}(\tau)\, d\tau \nonumber \\
& =  \frac{(-1)^{\lambda+2}}{(\lambda+1)!}\,\int_0^\pi 
F_n^\lambda (\tau)  \; g^{(\lambda+2)}(\tau)\, d\tau   . \qquad \qquad \qquad \Box
\end{align*}

\medskip
{\em Proof of  Theorem~\ref{th:polya_criterion}:}

Restrict attention, for the moment, to the case of  $d \in \{4,6,8\}$.

In order to apply the lemma we first have to mollify the function $g$.

\medskip
Given $g \in C^\lambda [0,\pi)$, which is zero in a neigbourhood of $\pi$, we extend the definition of
$g$ to $[0,\infty)$ by taking $g(x)=0$ for all $x \geq \pi$. 
For $h > 0$ define $G_h :[0,\infty) \rightarrow \R$ by 
\begin{equation} \label{eq:defGh}
G_h(x)= \frac{1}{h^2} \int_0^h \int_0^h g(x+u+v) \, du \, dv.
\end{equation}
Differentiating
$$
G_h^{(\lambda)} (x) = 
\frac{1}{h^2} \int_0^h \int_0^h g^{(\lambda)} (x+u +v) \, du \, dv\\
 = \frac{1}{h^2} \int^{x+h}_x \, \int_0^h g^{(\lambda)} (w+u) \, du \, dw.s
$$
Hence, 
\begin{equation} \label{eq:gh_lambdap1}
G_h^{(\lambda+1)}(x) = \frac{1}{h^2} \int_0^h \triangle_h \, g^{(\lambda)}(x+u)\, du, 
\end{equation}
where $\triangle_h$ is the usual forward difference operator. It follows that $ G^{(\lambda+1)}_h$
 is continuous on
$[0,\infty)$. Further, rewriting (\ref{eq:gh_lambdap1}) as
$$
G_h^{(\lambda+1)} (x) = \frac{1}{h^2} \int_x^{x+h} \left[ g^{(\lambda)}(w+h) - 
g^{(\lambda)}(w) \right] \, dw,
$$ and differentiating,
\begin{equation} \label{eq:gh_lambdap2}
G_h^{(\lambda+2)}(x) = \frac{1}{h^2}  \triangle^2_h \, g^{(\lambda)}(x).
\end{equation}
Hence $G^{(\lambda+2)}_h$ is also continuous on $[0, \infty)$.

\medskip
Clearly $G_h$ is supported in $[0,\pi)$ for all sufficiently small $h>0$. Also, it follows from the uniform continuity of $g$, and (\ref{eq:defGh}),
that $\{ g_h \}$ converges uniformly to $g$ on $[0,\pi]$, as $h \rightarrow 0^+$.
Hence, fixing  $n \in \bbN_0$ the  coefficients $\{b_n(G_h)\}_{h>0}$ converge to $b_n = b_n(g)$, as $h\rightarrow 0^+$.
By hypothesis $\psi:= (-1)^\lambda g^{(\lambda)}$ is a convex function. Therefore, from
(\ref{eq:gh_lambdap2}), $ (-1)^{\lambda+2} G_h^{(\lambda+2)}$  is a nonnegative function.
Hence, since $F_n^\lambda$ is nonnegative, it follows from Lemma~\ref{lem:alt_exp_Gegenbauer_coeffs}
that the coefficient, $b_n = b_n(g)$, and therefore the Gegenbauer coefficient $a_n = b_n/h_n$,
is nonnegative for all $n \in \bbN_0$.

\medskip
We now turn to the question of the boundedness, or otherwise, of $\sum a_n C^\lambda_n(1)$. 

\medskip
We need some standard facts about convex functions, which can be found, for example, in \cite{Royden}.
The function $\psi $ is convex on $[0,\pi)$. Therefore, it is absolutely continuous on 
every closed subset of $(0,\pi)$. The right and
left derivatives of $\psi$ exist at all points, $x\in(0, \pi)$, and are non-decreasing functions of $x$. These two one sided derivatives are continuous except at countably many points, and are equal at any point where one of them is continuous.  Thus the two sided derivative of $\psi$ exists, and is continuous, except 
at countably many points. 

\medskip
By hypothesis the one sided derivative $\psi'(0)$ also exists, and from the convexity
it is a lower bound for the value of $\psi'(x)$ on $[0, \pi)$.  Since $\psi'(\pi)=0$,  $\psi'(x) \leq 0$ 
for all $x \in [0,\pi)$ for which it is defined.
Using the absolute continuity of $\psi$, and defining $\eta_h = (-1)^\lambda G_h^{(\lambda)}$, (\ref{eq:gh_lambdap1})
can be rewritten
$$
   \eta_h'(x) = \frac{1}{h^2} \int_0^h \int_0^h \psi'(x+v+u) \, dv \, du.
$$
Thus $\eta_h'(x)$ is a weighted average of $\psi'$ over the interval $(x, x+2h)$.
Hence, for $0\leq  x$,
\begin{subequations}
\begin{equation} \label{eq:psip_le_etap}
 \psi'(x) \leq \eta'_h(x), \qquad \text{if $\psi'(x)$ exists}, 
\end{equation}
and
\begin{equation} \label{eq:etap_le_psip}
\eta_h'(x) \leq \psi'(x+2h), \qquad \text{if  $\psi'(x+2h)$ exits}.
\end{equation}
\end{subequations}
Then, for $h > 0$,
$$ \int_0^\pi \eta_h''(x) dx = \eta_h'(\pi) - \eta_h'(0) = 0 - \eta_h'(0) \leq - \psi'(0) = (-1)^{\lambda+1} g^{(\lambda+1)}(0).
$$

\medskip
Using Lemma~\ref{lem:alt_exp_Gegenbauer_coeffs}
it follows that
$$ 0 \leq b_n(G_h) \leq \frac{\|F_n^\lambda \|_\infty  | g^{(\lambda+1)}(0)|}{(\lambda+1)!}, $$
where $\| \cdot \|_\infty$ is the maximum norm on $[0,\pi]$.
Taking the limit as $h\rightarrow 0^+$ ,
$$ 0 \leq b_n(g) \leq  \frac{\|F_n\|_\infty  | g^{(\lambda+1)}(0)|}{(\lambda+1)!}. $$
Lemma~\ref{lem:FnbBound} shows that for $\lambda \in \bbN$,
 $\|F_n^\lambda\|_\infty = {\mathcal O}(n^{-3})$ as $n \rightarrow \infty$.
Also, from  \cite[(22.2.3)]{Abromowitz},  
$$ 
h_n = \frac{\pi 2^{1-2\lambda} \Gamma(n+2\lambda)}{n! (n+\lambda) [\Gamma(\lambda)]^2} 
\approx n^{2\lambda-2} \qquad \text{and} \qquad
C^\lambda_n(1) = \binom{n+2\lambda -1}{  n} \approx n^{2\lambda-1}.
$$
It follows from $a_n = b_n/h_n$, and the above, that $|a_n| C^\lambda_n(1) = {\mathcal O}(n^{-2})$.
Hence the Gegenbauer series of $g$ converges with $ \sum |a_n(g)| C^\lambda_n(1) < \infty $.
Combining the convergence of the series with the nonnegativity of  the Gegenbauer coefficients, shown previously, the first part of  Theorem~\ref{th:polya_criterion}
now follows as an application of Theorem~\ref{th:schoenberg}.

\medskip
We now turn to the part of the statement of   Theorem~\ref{th:polya_criterion} concerning strict postive definiteness. 

\medskip

We choose points $0<2a<b <\pi$ so that $\psi$,  restricted to $[2a,b]$, is not a linear polynomial, and
$\psi'(x)$ exists at $x=2a$, and also at $b$. Therefore, $\psi'(2a) < \psi'(b)$. Choose $h_1 >0$ so that
$a< 2a-2h_1  <b+2h_1 < \pi$ and $0 < h < h_1$.

It follows  from (\ref{eq:psip_le_etap}) and (\ref{eq:etap_le_psip}) that 
$ \eta_h'(a-2h)  \leq \psi'(2a)$ and
$\psi'(b)\leq \eta_h'(b)$.
Hence,
$$ \int_a^b \eta_h''(x)\, dx \geq \int_{2a-2h}^b \eta_h'' (x) \, dx = 
\eta_h'(b) - \eta_h'(2a-2h) \geq \psi'(b)-\psi'(2a) >0.
$$
Also $F_n^\lambda(\tau)$ is continuous and positive on the interval $[a, \pi]$.
 Thus, there is a number $\delta_n >0$
such that $F_n^\lambda(\tau) \geq \delta_n$ for all $\tau$ in $[a, \pi]$. 
An application of Lemma~\ref{lem:alt_exp_Gegenbauer_coeffs} now shows that for each $n \in \bbN_0$,  
$b_n(G_h) \geq \delta_n \left\{ \psi'(b)-\psi'(2a)\right\} > 0$. 
Since $0 < h < h_1$ was arbitary, taking the limit, as $h \rightarrow 0^+$, shows
$b_n(g) >0$. Hence all the Gegenbauer coefficients, $a_n = b_n/h_n$,  of $g$ are positive, and the sufficient condition of \cite{Xu92}, discussed in the introduction, shows that $g$ is strictly positive definite on $\bbS^{d-1}$.

\medskip
The above has established both parts of Theorem~\ref{th:polya_criterion} when
$d \in \{4,6,8\}$. However, a function $g=f (\cos \cdot)$ which is (strictly)
positive definite on $\Sdmone$ is necessarily (strictly) positive definite on ${\mathbb S}^{d-2}$. Hence,
Theorem~\ref{th:polya_criterion} for $d\in \{ 4, 6,8\}$ implies
Theorem~\ref{th:polya_criterion} for $d \in\{3,5,7\}$.  \hfill{$\Box$}

\smallskip
We note that if the result of Proposition~\ref{propn:proven_cases} were available for more values
of $d \in 2 \N$ then the above proof would immediately give Theorem~\ref{th:polya_criterion} for more values of $d$.

\medskip
Let us also note that increasing the power on $(t-\cdot)_{+}$ in the function
$g(\cdot)=(t-\cdot)^\delta_{+}$ will preserve any existing positive definiteness. More precisely, 
\begin{lem} \label{lem:reduction}
If $F_n^{\lambda, \delta}(t)$ is positive on $(0,\pi]$, except possibly for finitely many points $t$,
then, for all  $\mu > \delta$,  $F_n^{\lambda, \mu}(t)$
is positive on $(0,\pi]$. 
\end{lem}
\begin{proof}
This follows from the  the semigroup property, $\CL^{\d+\mu} = \CL^\d \CL^\mu$,
for the fractional integrals
$$
  \CL^\delta f :=  \frac{1}{\Gamma(\d)} \int_0^t (t- \t)^{\d-1} f(\t) d\theta.  \qedhere
$$ 
\end{proof}
This lemma shows the pivotal role of the boundary case involving $F^\lambda_n(t)$.

\section{Proofs of the conjecture in low dimensional cases}

In this section we present proofs of the conjecture in low dimensional cases as detailed in 
Proposition~\ref{propn:proven_cases}.

\medskip

Recall the formula connecting Gegenbauer polynomials with different parameters, \cite[p. 99]{Sz75}. For $\mu > (\l-1)/2$, 
\begin{equation}\label{connection}
   (\sin \t)^{2\mu}\, C_n^\mu(\cos \t) =\sum_{k=0}^\infty c_{k,n}^{\mu,\l}  
       (\sin \t)^{2 \l}\, C_{n+2k}^\l (\cos \t), 
\end{equation}
where
\begin{equation}\label{connection2}
   c_{k,n}^{\mu,\l} = \frac{2^{2\l-2\mu} \Gamma(\l)\Gamma(n+2\mu)}
        {\Gamma(\mu)\Gamma(\l-\mu)}
        \frac{ (n+2k+\l) (n+2k)! \, \Gamma(n+k+\l)\Gamma(k+\l-\mu)}
        {n! \, k! \, \Gamma(n+k+\mu+1)\Gamma(n+2k+2\l)}.
\end{equation}
Note that if $\lambda-\mu$ is a negative integer the connection coefficient $c^{\lambda,\mu}_{k,n}$  is only nonzero
for $0 \leq k \leq \mu - \lambda$.

\begin{lem} \label{lem:cosine}
For $\mu \geq 1$, 
\begin{equation}\label{cosine}
     C_n^\mu (\cos \t) (\sin \t)^{2\mu} = \sum_{k=0}^\infty c_{k,n}^\mu \cos (n + 2 k) \t,
\end{equation}
where 
\begin{align*}
 c_{k,n}^\mu: =  \frac{2^{1-2\mu}(-\mu)_k \Gamma(n+2\mu)\Gamma(n+k) (n+2k)}
          {\Gamma(\mu) n!\, k!\, \Gamma(n + k+\mu+1)}.
\end{align*}
When $\mu \in \NN$ the summation terminates 
at $k=\mu$ and the expression for $c^\mu_{k,n}$ can be rewritten as
\begin{equation} \label{eq:c_mu_k_n_alt} 
c^\mu_{k,n} \; = \; \frac{ 2^{1-2\mu}}{\Gamma(\mu)}  (- 1)^k \binom{\mu}{k}  \frac{ (n+1)_{2\mu-1} (n+2k ) }{(n+k)_{\mu+1}}.
\end{equation}
\end{lem}

\begin{proof} 
Recall that $C_m^1(\cos \t) = U_n(\cos \t) = \sin (n+1)\t/\sin \t$. From the case 
$\l =1$ of \eqref{connection} we deduce that 
\begin{align*}
 C_n^\mu (\cos \t) (\sin \t)^{2\mu} & = \sum_{k=0}^\infty c_{k,n}^{\mu,1} 
    \sin(n+2k+1)\t \sin \t \\
    & = \frac12 \sum_{k=0}^\infty c_{k,n}^{\mu,1}  (\cos(n+2k)\t - \cos (n+2k+2)\t \\
    & = \frac12 \sum_{k=0}^\infty (c_{k,n}^{\mu,1} - c_{k-1,n}^{\mu,1}) \cos(n+2k)\t, 
\end{align*}
so that 
$c_{k,n}^\mu = \frac{1}{2}c_{k,n}^{\mu,1} - \frac{1}{2}c_{k-1,n}^{\mu,1}$,
 the explicit formula of which 
is deduced from \eqref{connection2}. 
\end{proof}

For $\mu$  a positive integer,  equation \eqref{cosine} can also be deduced from the 
following relation for Gegenbauer polynomials, 
\begin{equation}\label{eq2}
    (1-x^2) C_k^{\l+1} (x) = \frac{(k+2\l+1)(k+2 \l)}{4\l (k+\l+1)} C_k^\l(x)
        - \frac{(k+2)(k+1)}{4\l (k+\l+1)} C_{k+2}^\l (x). 
\end{equation}
The equation \eqref{cosine} allows us to write down an explicit formula for $F_k^\l(t)$.

\begin{lem} \label{lem:Fn}
For  $\mu =1,2,3, \ldots$, 
\begin{align*}
&   F_n^{2\mu-1}(t) = \sum_{k=0}^{2 \mu-1} c_{k,n}^{2\mu-1} 
   \frac{(-1)^{\mu} (2\mu)!}{(n+2k)^{2\mu+1}} \left[ \sin (n+2k) t - \sum_{j=0}^{\mu-1} (-1)^j 
       \frac{(n+2k)^{2j+1}}{(2j+1)!}  t^{2j+1} \right], \\
&  F_n^{2\mu}(t) = \sum_{k=0}^{2 \mu} c_{k,n}^{2\mu} 
   \frac{(-1)^{\mu+1} (2\mu+1)!}{(n+2k)^{2\mu+2}} \left[ \cos (n+2 k) t - \sum_{j=0}^\mu (-1)^j 
  \frac{(n+2k)^{2j}}{(2j)!} t^{2j} \right].
\end{align*}
\end{lem}

\begin{proof}
It follows from Taylor's theorem with integral remainder that
\begin{align*}
 \cos k t - \sum_{j=0}^\mu (-1)^j 
  \frac{(k t)^{2j}}{(2j)!}
& \ = \ (-1)^{\mu+1} \frac{k ^{\l+2}}{(\l+1)!}\int_0^t (t-\t)^{\l+1} \cos k \t \, d\t , \quad \l = 2 \mu,  \\
 \sin k t - \sum_{j=0}^{\mu-1} (-1)^j  
  \frac{(k t)^{2j+1}}{(2j+1)!} 
& \ =\   (-1)^{\mu} \frac{k ^{\l+2}}{(\l +1)!}  \int_0^t (t-\t)^{\l+1} \cos k \t \, d\t , \quad \l = 2 \mu-1. 
\end{align*}
Consequently, together with the identity in  Lemma~\ref{lem:cosine}, we obtain 
an explicit formula for $F_n^\l(t)$. 
\end{proof}

\extras{
More explicitly, for the first few values of $\lambda$   Lemma~\ref{lem:Fn} yields
\begin{equation} \label{eq:F1}
 F_n^1(t) =  \sum_{k=0}^1
 d_{k,n} \left\{ \sin ((n+2k)t) -  (n+2k)t  \right\},
\end{equation}
where
$$ d_{0,n} = \frac{-1}{n^3} \qquad \text{and} \qquad d_{1,n}= \frac{1}{(n+2)^3},
$$

 \begin{rem} \label{Fn_space} For $\mu \in \bbN$ define $ {\mathcal S}_{2\mu-1,n} =\{ \sin(nt),  \sin((n+2)t), \ldots , \sin((n+4\mu-2)t), t, t^3, \ldots, t^{2\mu-1}\}$, and ${\mathcal  S}_{2\mu,n}=\{\cos(nt), \cos((n+2)t), \ldots, \cos((n+4\mu)t), 1, t^2, \ldots t^{2\mu}\}$. 
It follows from Lemma~\ref{lem:Fn} and equation~(\ref{eq:c_mu_k_n_alt}) that
$F^{\lambda}_n$ lies in $\text{span}\left({\mathcal S}_{\lambda,n}\right)$.
Also the
coefficients  which occur in expressing $F^\lambda_n$ in terms of the functions in ${\mathcal S}_{\lambda,n}$,
are rational functions of $n$.
\end{rem}
}

Prototypical special cases are,
\begin{equation}
  \frac{4}{3}\,F_n^2(t)  =  \sum_{k=0}^2 \frac{e_{k,n}}{ (n+2k)^4}
\left\{ \cos((n+2k)t)- \left[ 1 -\frac{(n+2k)^2 t^2}{2}\right] \right\}  \label{eq:F2},
\end{equation}
where
$$
e_{0,n}= (n+3), \quad e_{1,n}= -2(n+2) \quad \text{and}
\quad e_{2,n}= (n+1),
$$
and
\begin{align}
  \frac{8}{3}\,F_n^3(t) & =  \sum_{k=0}^3
\frac{f_{k,n}}{(n+2k)^5} \left\{ \sin ((n+2k)t) - \left[ (n+2k)t - \frac{(n+2k)^3 \,  t^3}{6} \right] \right\},
\label{eq:F3} \end{align}
where
$$
f_{0,n} = (n+5)(n+4), \ f_{1,n}= -3 (n+5)(n+2), \ f_{2,n}= 3(n+4)(n+1)\ \text{and}
\ f_{3,n}= -(n+2)(n+1).
$$

\bigskip 

Let us write \eqref{eq2} as  
\begin{equation}\label{eq3}
    \frac{2(k+\l+1)}{2\l+1} (1-x^2)\frac{C_k^{\l+1}(x)}{C_k^{\l+1}(1)} 
   =   \frac{C_k^{\l}(x)}{C_k^{\l}(1)} - \frac{C_{k+2}^{\l}(x)}{C_{k+2}^{\l}(1)}.
\end{equation}
For $0 \le \d \le \l$ define 
\begin{equation}\label{Gnlambda}
  G_n^{\l,\d}(t): = \frac{F_n^{\l,\d+1}(t)}{C_n^\l(1)} = \int_0^t (t-\t)^{\d+1} \frac{C_n^\l(\cos \t)}{C_n^\l(1)} (\sin \t)^{2\l}  d\t. 
\end{equation}
Then the following relation follows immediately from \eqref{eq3}: 

\begin{lem} 
For $\l \in \NN$ and $0\le \d \le \l$, 
\begin{equation}\label{inductG}
    \frac{2}{2\l-1}  G_n^{\l,\d}(t) = \frac{1}{n+\l} 
         \left[G_n^{\l-1,\d}(t) - G_{n+2}^{\l-1,\d}(t)\right]. 
\end{equation}
\end{lem}

Let us start from $\l =0$ and make a change of variable in the integral to obtain
$$
   G_n^{0,\d}(t)  = \int_0^t (t-\t)^{\d+1} \cos n \t d\t  \\
         =  t^{\d+2}\int_0^1(1-s)^{\d+1} \cos (n t s) ds = t^{\d+2} h_\d(n t), 
$$
where 
\begin{equation} \label{eq:h_delta}
     h_\d(u) := \int_0^1(1-s)^{\d+1} \cos (u s) ds. 
\end{equation}

Let us define
\begin{equation} \label{eq:H_1}
 H_1^\d(n, \xi_1,t): = h_\delta'((n+\xi_1)t)
\end{equation}
 and, for $j =1,2,\ldots$, define inductively 
\begin{equation} \label{eq:defHj}
   H_{j+1}^\d (n,\xi_{j+1}, \xi_j, \ldots, \xi_1,t) : =  \frac{\partial}{\partial \xi_{j+1}} 
        \left[\frac{H_{j}^\d(n+\xi_{j+1}, \xi_j, \ldots , \xi_1,t)}{n+j+\xi_{j+1}} \right].
\end{equation}

\begin{lem}
For $j = 1, 2, \ldots, $, 
\begin{align} \label{induct-main}
   \frac{2^j}{(2j -1)!!} G_n^{j,\d}(t)   = (-1)^{j} \frac{t^{\d+3}}{n+j}  
      \int_{[0,2]^j} H_{j}^\d(n,\xi_j ,\ldots, \xi_1,t) d\xi_j \cdots d\xi_1.  
\end{align}
\end{lem}

\begin{proof}
Applying \eqref{inductG} with $\l=1$ shows that 
\begin{equation} \label{eq:Gn_1_delta}
2 G_n^{1,\d}(t) =  \frac{t^{\d+2}}{n+1} \left[h_\delta(nt) - h_\delta ((n+2)t)\right] = 
  -  \frac{t^{\d+3}}{n+1} \int_0^2 h_\delta'((n+\xi)t) d\xi,
\end{equation}
which proves \eqref{induct-main} when $j =1$. For $j > 1$ we use induction and \eqref{inductG} to conclude 
\begin{align*}
\frac{2^{j+1}}{(2j +1)!!} & G_n^{j+1,\d}(t)   = \frac{ (-1)^{j} t^{\d+3} }{n+ j+1} \\
   &\times \int_{[0,2]^j} \left[ \frac{ H_{j}^\d(n,\xi_j ,\ldots, \xi_1,t) } {n+j}  
        - \frac{ H_{j}^\d(n+2,\xi_j ,\ldots, \xi_1,t) } {n+2+j} \right]  d\xi_j \cdots d\xi_1,
\end{align*}
from which the \eqref{induct-main} for $j+1$ follows readily. 
\end{proof}

In particular, for $j =2$, this gives 
\begin{align}  \label{induct2}
 \frac{4}{3} G_n^{2,\d}(t)  =  \frac{t^{\d+3}}{n+2} \int_0^2 \int_0^2 \frac{\partial}{\partial \xi_2 } 
 \left[\frac{h_\delta'((n+\xi_1+ \xi_2)t) }{n+1+\xi_2} \right]   d\xi_2 \,  d\xi_1.
\end{align}
 
\begin{lem}
For $r =0,1,2,...\l$, 
\begin{equation} \label{eq:h_lambda}
     |h_\l^{(r)}(u)| \le c \, u^{- r -2}, \qquad \text{for}  \   u \ge 1, 
\end{equation}
where the constant $c$ depends only on $\l$ and $r$.
\end{lem} 
 
\begin{proof}
Let us denote by $g_\l$ the function
$$
 g_\l (u) : =  (\l+1)!
     \begin{cases} \displaystyle{(-1)^{\mu+1} \biggl [ \cos u - \sum_{j=0}^\mu (-1)^j  \frac{u^{2j}}{(2j)!} \biggr] }, & \l = 2 \mu, \\ 
         \displaystyle{ (-1)^\mu \biggl [\sin u - \sum_{j=0}^{\mu-1} (-1)^j  \frac{ u^{2j+1}}{(2j+1)!} \biggr]}, & \l = 2 \mu-1.
           \end{cases}
$$
From the displayed identities in the proof of Lemma 3.2, we then obtain
$ \displaystyle 
   h_\l(u) =   \frac{g_\l(u)}{u^{\l + 2}}. 
$ 
Since it is evident that $|g_\l(u)|  \le c \ u^\l$ for $u \ge 1$, it follows that $|h_\l(u)| \le c u^{-2}$. Taking derivatives, it is easy
to see that $|g_\l ^{(j)}(u)| \le c u^{\l -j}$ for $1 \le j \le \l$ and $u \ge 1$. Consequently, by the Leibniz rule, 
$$
   h_\l^{(r)}(u) = \sum_{j=0}^r \binom{r}{j} (-1)^{j} (\l+2)_j u^{- (\l+2 +j)}g_\l^{(r-j)}(u),
$$
from which the stated estimate follows. 
\end{proof}

\begin{lem} \label{lem:FnbBound}
Let $\|\cdot\|_\infty$ denote the uniform norm on $[0,\pi]$. For $\l = 1,2, \ldots$, 
$$
          \|F_n^\l\|_\infty = \CO \left( n^{-3}\right), \quad \text{as}\ n \to \infty.
$$
\end{lem} 

\begin{proof} Throughout this proof $c$  represents a constant, possibly different at every occurrence, depending only on $\lambda$.

If $n t \le 1$, then we use the expression of $F_n^\l$ in Lemma 3.2, in which the terms in the square brackets are
bounded by an absolute constant depending only on $\l$. Thus, the estimate $F_n^\l(t)| \le c n^{-3}$ for $0 \le t \le n^{-1}$
follows immediately from the fact that the coefficients $c_{k,n}^\mu$ satisfies $|c_{k,n}^\mu| = \CO(n^{\mu -1})$, see \eqref{eq:c_mu_k_n_alt}. 

We now assume $ 0 < t \leq  \pi$, $n t \ge 1$, and $\xi_i \in [0,2]$, for all $i$. Since $F_n^\l = C_n^\l(1) G_n^{\l,\l} = \binom{n+ 2\l-1}{n} G_n^{\l,\l}$,  it is sufficient 
to show that  $\max_{t \in [n^{-1}, \pi]} |G_n^{\l,\l}(t)\| \leq c \, n^{- 2 \l - 2}$. We claim 
that for $1 \leq j \leq \l$ the kernel $H_j^\l$ of 
$G_n^{j,\l}$ in \eqref{induct-main} is of the form
\begin{equation} \label{eq:Hj}
  H_j^\l (n,\xi_j,\ldots, \xi_1,t) = \sum_{i=1}^j R_{i,j} (n,\xi_j,\ldots, \xi_1,t)\; h_\lambda^{(i)}((n+\xi_j+\ldots+\xi_1)t),
\end{equation}
where the $R_{i,j}$ are rational functions  of the form 
\begin{equation} \label{eq:Rij}
   R_{i,j} = \frac{P_{i,j}}{Q_{i,j}}, \qquad  \deg( R_{i,j}) := \deg P_{i,j}- \deg Q_{i,j}  \leq 1+i -2j, 
\end{equation}
where the $P_{i,j}$ are polynomials in $t,n, \xi_2, \ldots, \xi_j$;  the $Q_{i,j}$ are polynomials in 
$n, \xi_2, \ldots \xi_j$; and their degrees refer to their highest degree in $n$. Furthermore, 
\begin{equation}\label{eq:prop_of_Qij} \text{
the coefficient
of the highest power of $ n$ in the polynomial $Q_{i,j}$ can be chosen as $1$}.
\end{equation}

Assume for now that \eqref{eq:Hj}, \eqref{eq:Rij} and \eqref{eq:prop_of_Qij} have been shown. We then have  
$$
   |R_{i,j}(n,\xi_j,\ldots, \xi_1,t) | = \CO \left( n^{-(2j-i-1)}\right), 
$$ 
Since $n t \ge 1$ implies that $(n+\xi_j + \cdots + \xi_1)t \ge 1$, \eqref{eq:h_lambda} and \eqref{eq:Hj} then imply, 
$$
  |H_j^\l(n,\xi_j,\ldots, \xi_1,t) | = \CO \left( \sum_{i=1}^j n^{-(2j-i-1)} (n t)^{- i -2}  \right) \le 
     c \,t^{-j -2} n^{-2j-1},
$$
for $1\leq j \leq \lambda$, where  $0 \leq \xi_1 , \ldots , \xi_j \leq 2$. Consequently, 
from \eqref{induct-main} with $j = \l$ follows  
$$
\max_{t\in[n^{-1},\pi]} \left| G_n^{\l,\l}(t) \right|  \le 
\max_{t \in [n^{-1},\pi]}  c \, t^{\l+3} t^{-\l -2} n^{- 2 \l - 1} \le c\, n^{-2 \l -1},
$$ 
which shows that $\max_{t\in [n^{-1},\pi]} \left| F_n^\l (t) \right| = \CO \left( n^{-3}\right)$, and thus
$\|F_n^\l \|_\infty = \CO(n^{-3})$. 

\medskip
It only remains to prove  \eqref{eq:Hj}, \eqref{eq:Rij} and \eqref{eq:prop_of_Qij}. The proof is  by induction on $j$. 

{\em Induction basis:} In the case $j=1$ the identities
follow from the definition (\ref{eq:H_1}), and the choice $P_{1,1}=Q_{1,1}=1$. 

{\em Induction step:} Assume that the properties
have been established up to $j=k$. We shall leave out the argument of $H_{k+1}$ and $R_{i,k+1}$ below 
and trust that no confusion is likely to occur.  By \eqref{eq:defHj} and the induction hypotheses, 
$$
   H_{k+1}^\l  = \sum_{i =1}^k \frac{\partial}{\partial \xi_{k+1}}
     \left[ \frac{R_{i,k} (n+\xi_{k+1} ,\xi_k,\ldots, \xi_1,t)}{n+k+\xi_{k+1}} h_\lambda^{(i)}((n+\xi_{k+1}+\ldots+\xi_1)t) \right],
$$
from which we immediately deduce that \eqref{eq:Hj} holds for $H_{k+1}^\l$ with 
\begin{equation} \label{eq:Rjj-induct}
  R_{k+1,k+1} = t \frac{R_{k,k} (n+\xi_{k+1} ,\xi_k,\ldots, \xi_1,t)}{n+k+\xi_{k+1}},
\end{equation}
and, for $1 \le i \le k$, with $R_{0,k}:=0$,
\begin{align} \label{eq:Rij-induct}
 R_{i,k+1} & =  \frac{\partial}{\partial \xi_{k+1}}
 \left[ \frac{R_{i,k} (n+\xi_{k+1} ,\xi_k,\ldots, \xi_1,t)}{n+j+\xi_{k+1}}\right] + t \frac{R_{i-1,k} (n+\xi_{k+1} ,\xi_k,\ldots, \xi_1,t)}{n+k+\xi_{k+1}} \notag \\
  & =: R_{i,k+1}^{(1)} + R_{i,k+1}^{(2)}. 
\end{align}
It follows from \eqref{eq:Rjj-induct}, by the induction hypotheses, that  $\deg( R_{k+1,k+1}) = \deg( R_{k,k}) -1 \leq -k$. 
For $ 1 \le i \le k$,  quick computations show that, by  the induction hypotheses, 
$$
 \deg (R_{i,k+1}^{(1)}) = \deg (R_{i,k}) -2  \leq i-1-2k, \quad  \deg (R_{i,k+1}^{(2)}) = \deg (R_{i-1,k}) - 1 
\leq i -1 -2k.
$$
Finally, if $R = R_1 + R_2$ and $\max (\deg R_1, \deg R_2 ) = \ell$, then placing $R_1$ and $R_2$ over a common denominator, 
$\deg R \leq \ell$. Consequently, by \eqref{eq:Rij-induct}, $\deg R_{i,k+1} \leq i -1 -2k$. 
This shows \eqref{eq:Rij} for $1 \leq i \leq k+1$ and $j=k+1$. Finally,  \eqref{eq:prop_of_Qij} for $j=k+1$ and $1\leq i \leq k+1$ follows from  \eqref{eq:Rjj-induct}, \eqref{eq:Rij-induct} and the induction hythotheses.
Thus, if the three properties hold up to $j=k$, they also hold for $j=k+1$.

{\em Conclusion:} The result follows by induction for all positive integers $i$ and $j$, with $1\leq i \leq j$. 
\end{proof}

\medskip

{\em Proof of the case $d=4$, that is $\lambda =1$, of Proposition~\ref{propn:proven_cases}.}

\medskip

The case $n=0$ is trivial. Assume now that $n>0$.
Recalling that\\
 $U_n(\cos(\theta))=\sin((n+1)\theta)/\sin(\theta)$
we have 
\begin{align}
F_n^1 (t) &= \int_{\theta=0}^t \left(t -\theta \right)^2 U_n \left( \cos(\theta) \right) \sin^2 \theta 
   d\theta   \notag \\
&=\int_0^t \left(t -\theta \right)^2 \sin \left( (n+1) \theta \right) \sin \theta \, d\theta \notag \\
& = \frac12 \int_0^t \left(t-\theta \right)^2 \left\{ \cos(n\theta) - \cos( (n+2)\theta) \right\} d\theta
 \notag \\
& =  I(t,n)-I(t,n+2),  \label{HandI}
\end{align}
where 
$$ 
  I(t,m) = \frac12  \int_0^t\left( t - \theta \right)^2 \cos( m\theta) \, d\theta. 
$$
Integrating by parts shows that 
\begin{equation} \label{eq:posItm}
I(t,m) = t^3 h(m t ) \quad \hbox{with} \quad  h(u) =  \frac{ u - \sin u}{u^3}, \quad u > 0.
\end{equation}
Therefore
\begin{align} \label{derivI_a_}
\frac{\partial}{\partial m} I(t,m)   = - t^4 h'(m t) = - t^4 \; \frac{2 u + u \cos u -3 \sin u}{u^4},  
\end{align}
where $u=m t$.  By the trivial inequalities $|\cos u| \leq 1$ and $|\sin u| \leq 1$, we see  
that 
$$
 g(u):=  2 u + u \cos u -3 \sin u \ge   2 u  - u -3 = u - 3 > 0, \qquad \text{for}\ u>3.
$$
On the other hand, the Taylor expansion of $g(u)$ takes the form
$$
   g(u) = 2 \sum_{k=2}^\infty (-1)^k \frac{ (k-1) u^{2k+1}}{(2k+1)!}
       = \frac{2 u^5}{5!} -  \frac{4 u^7}{7!} +  \frac{6 u^9}{9!} -  \frac{8 u^{11}}{11!} + \ldots
$$
which is an alternating series, hence positive, if $0< u^2 < 21$, which clearly covers 
$u \in (0,3]$. Consequently, it follows that $g(u) > 0$ for $u > 0$, which implies that for $t>0$,
$h( m t)$, hence $I(t,m)$, is strictly decreasing in $m > 0$. Therefore, by (\ref{HandI}), 
$F_n^1(t)$ is positive for all $t>0$ and all positive integers $n$. \hfill $\Box$
 
\medskip
{\em Proof of the case $d=6$, that is $\lambda =2$, of Proposition~\ref{propn:proven_cases}.}

\medskip
The proof below splits into two cases, $t$ near zero, and  $a/n < t \leq \pi$, 
where $a$ is a constant yet to be determined.

We deal with the second case first.
Recall from \eqref{Gnlambda} that $F_n^\lambda(t) = C_n^\lambda(1) G_n^{\lambda,\lambda}(t)$,
and from \eqref{inductG},
$$
    \frac{2}{2\l-1}  G_n^{\l,\d}(t) = \frac{1}{n+\l} 
         \left[G_n^{\l-1,\d}(t) - G_{n+2}^{\l-1,\d}(t)\right]. 
$$
We need to show that $G_n^{2,2}(t) > 0$, which holds if $\frac{d}{dn}G_n^{1,2}(t) < 0$.

From \eqref{eq:Gn_1_delta}
$$
2\, G_n^{1,2}(t) =   
  -  \frac{t^{5}}{n+1} \int_0^2 h_2'((n+\xi)t) d\xi,
$$
where $h_2$ is defined in  \eqref{eq:h_delta}. Therefore, writing $h$ for $h_2$, we have
$$
2t^{-5} \frac{d}{dn} G_n^{1,2}(t) = - \int_{0}^2  \left[ \frac{ t\, h''((n+ \xi)t)}{n+1} - \frac{h'((n+\xi)t)}{(n+1)^2}
        \right] d\xi   .   
$$

Our immediate goal is to find a constant $a$ such 
that $\frac{d}{dn} G_n^{1,2}(t) < 0$ for $a/n < t \leq \pi,$ which will allow us
to conclude that $F^2_n(t) > 0$, for $a/n < t \leq \pi$. Evidently, it is sufficient
to show that the function being integrated above is positive. Hence, we see that 
$\frac{d}{dn} G_n^{1,2}(t) < 0$ if 
\begin{equation} \label{Hu}
  H(u):=  (n+1) t\, h ''((n+\xi)t) - h'((n+\xi)t ) = \frac{n+1}{n+\xi} u h''(u) - h'(u)  > 0,
\end{equation}
for all $ 0 < \xi < 2$, where $u = (n+\xi)t$.
 
A simple calculation  shows that
$$ h(u)=h_2(u) = \frac{6}{u^4} \left( \cos(u) -1 + \frac{u^2}{2} \right).
$$

Next we show that for all $u>0$, $h'(u) < 0$. 
Taking derivatives and Maclaurin series we see that 
\begin{align*}
   h'(u) &= \frac{ 6 \left(4 - u^2 - 4 \cos u - u \sin u \right) }{u^5}  \\
          & = \frac{6}{u^5} \sum_{k=3}^\infty \frac{2 (-1)^k \, (k-2) }{ (2k)!} u^{2k}
           =  \frac{6}{u^5} \left( -\frac{u^6}{360} + \frac{u^8}{10080} - \frac{u^{10}}{604800} - \ldots\right),
\end{align*}
The series is an alternating series, and negative, for $u^2 < 28$, thus certainly for $0 < u \leq 4$.  Furthermore, when $u>4$ the $-u^2$ dominates the other terms in the numerator 
and $h'(u)$ is again negative. The desired result follows.

Now we consider 
$$
   h''(u)  = 
       \frac{6 (-20 + 3 u^2 - (-20 + u^2) \cos u  + 8 u \sin u )}{u^6}.
$$
The trivial inequalities $|\sin u| \leq 1$ and $|\cos u| \leq 1$ show that
$$
   h''(u) \ge 6 u^{-6}(3 u^2 - u^2 - 8 u - 40) = 12 u^{-6} (u^2 - 4 u - 20) > 0,
$$
if $u > 2(1+\sqrt{6}) = 6.898989...$\ . In fact, numerical computation shows that $h''(u)$ has 
a single root for $0 < u < 7$ at $u_0 = 3.68542 ....$. Thus, $h''(u)$ is positive if $u > u_0$.
Combining these facts about $h'$ and $h''$ we see that $H(u) > 0$ if $u > u_0$. 

We now consider the case $0 < u \leq u_0$ for which  $h''(u) \le 0$. Then  
$0 < \xi < 2$ implies by \eqref{Hu}
$$
  H(u) \ge \frac{n+1}{n}  u h''(u) - h'(u) \ge \frac{4}{3} u h''(u) - h'(u), \qquad 
\text{when} \ n \ge 3. 
$$
Define $k(u)$ to be the estimate of  $H(u)$ from below given above
$$
 k(u):=\frac{4}{3} u h''(u) - h'(u) = \frac{u^{-5}}{3} \left[-92+ 15 u^2 -4(-20+u^2) \cos u + 
35 u \sin u +12 \cos u\right]. 
$$
Applying the trivial inequalities $|\cos u| \leq 1$ and $|\sin u| \leq 1$ shows that $k(u)$
is positive if $u > 5.9793...  $. Numerical computation further shows that $k(u)$ has one 
single zero in $0 < u < 6$ at $u_1=1.86321... $. That is $k(u)$ is positive for all
$u > u_1$. 

Combining the results so far we have shown $H(u)>0$ for all $u > u_1$. In view of
the remarks near equation \eqref{Hu} this allows us to conclude that 
\begin{equation} \label{eq:lam2_away_from_zero}
F^2_n(t) > 0\ \text{whenever} \  u_1/n < t \leq \pi, \ \text{and} \ n \geq 3.
\end{equation}

We still have to analyze the behaviour of $F_n(t)$ when $0 < t \leq u_1/n$.
By an inequality in \cite{El90}, the largest zero of the Gegenbauer polynomial
$C_n^\l$ satisfies the inequality
$$
    x_{n,1}(\l) \le \frac{\sqrt{n^2+2(n-1)\l -1} }{(n+\l)}. 
$$ 
Taking $t_n^* =\arccos (x_{n,1}(2))$ it follows that
$$ \cos(t^{*}_n)=x_{n,1}(2) \leq \sqrt{ 1 -\frac{9}{(n+2)^2}} $$
so  that 
$$ (\sin t^{*}_n)^2 \geq \frac{9}{(n+2)^2},
$$ and $C^2_n(\cos \theta )$ is positive whenever $0 <\theta < t_n^* $, and therefore whenever
$0 < \theta \leq \sin t_n^* $.
By the definition of $F_n^2(t)$, 
this shows that
\begin{equation} \label{eq:lam2_near_zero}
F^2_n(t) > 0 \ \text{ whenever} \  0 < t \leq 3/(n+2).
\end{equation}
The regions of positivity given by equations (\ref{eq:lam2_away_from_zero}) and (\ref{eq:lam2_near_zero})
overlap when $n \geq 4 $, so that for all such $n$, $F^2_n(t) >0$ for all $0 < t \leq \pi$.
The same conclusion can be reached for $0\leq n \leq 3 $ by plotting the explicit expression
 (\ref{eq:F2}) for $F^2_n(t)$. \hfill $\Box$

\bigskip
{\em Proof of the case, $d=8$ that is $\lambda =3$, of Proposition~\ref{propn:proven_cases}.}

\bigskip
The proof below splits into two cases, $t$ near zero, and $t$  greater than $a/n$, as did the proof in the
$d=6$ case.

\medskip
We deal firstly with the case $t > a/n$. 
Recall from \eqref{Gnlambda} that $F_n^\lambda(t) = C_n^\lambda(1) G_n^{\lambda,\lambda}(t)$,
and from \eqref{inductG}
$$
    \frac{2}{2\l-1}  G_n^{\l,\d}(t) = \frac{1}{n+\l} 
         \left[G_n^{\l-1,\d}(t) - G_{n+2}^{\l-1,\d}(t)\right]. 
$$
We need to show that $G_n^{3,3}(t) > 0$, which holds if $\frac{d}{dn}G_n^{2,3}(t) < 0$.

Our immediate goal therefore is to find a constant $a$ so that
$\frac{d}{d n} G_n^{2,3}(t) < 0$ for all $a/n < t \leq \pi$. By
\eqref{induct2}  with $\delta=3$, and the function $h=h_3$, defined in \eqref{eq:h_delta},
$$
\frac{4}{3} G_n^{2,3}(t)= \frac{t^6}{n+2} \int_0^2 \int_0^2 \left[ 
    \frac{t h''((n+\xi+\eta)t)}{n+1+\eta} - \frac{h'((n+\xi+\eta)t)}{(n+1+\eta)^2} 
        \right] \, d\eta \, d\xi. 
$$
Taking the derivative with respect to $n$ and simplifying, we obtain
\begin{align*}
  \frac{4}{3}\frac{d}{dn} G_n^{2,3}(t) = t^6 \int_0^2 \int_0^2 & \left[ 
   \frac{t^2 h'''((n+\xi+\eta)t)}{(n+2)(n+1+\eta)} 
     - \frac{(3n+5+\eta) t h''((n+\xi+\eta)t)}{(n+2)^2 (n+1+\eta)^2} \right.  \\
     & \qquad \left.     +  \frac{(3n+5+\eta)  h'((n+\xi+\eta)t)}{(n+2)^2 (n+1+\eta)^3} 
        \right] \, d\eta \, d\xi. 
\end{align*}
In order to show that $\frac{d}{dn} G_n^{2,3}(t)  < 0$, we only need to show that the
integrand is negative, for all $0 < \xi, \eta <2$. We introduce a function
\begin{align} \label{Hn}
  H_n(u)  : =  u^2 h'''(u) & - \frac{(3n+ 5 + \eta)(n+ \xi + \eta)}{(n+2)(n+ 1 + \eta)} u h''(u) \\ 
    &  + \frac{(3n+ 5 + \eta)(n+ \xi + \eta)^2}{(n+2)(n+ 1 + \eta)^2} h'(u). \notag
\end{align}
Then it is easy to see that 
\begin{align} \label{Gn23}
\frac{4}{3}\frac{d}{dn} G_n^{2,3}(t) = t^6 \int_0^2 \int_0^2  a_n H_n( (n+\xi+\eta)t) \,d\eta \, d \xi,
\end{align}
where $a_n = a_n(\xi,\eta)= 1/((n+2)(n+1+\eta)(n+\xi+\eta)^2) > 0$. Thus, to demonstrate that
 $\frac{d}{dn} G_n^{2,3}(t)$
is negative, it is sufficient to show that
$H_n(u) < 0$ for all $0 < \xi, \eta <2$, where $u=(n+\xi+\eta)$. 

Now, a simple computation shows that 
$$
  h(u)= h_3(u) = \int_0^1( 1-s)^4 \cos (s u) ds = \frac{4 (-6 u + u^3 + 6 \sin u)}{u^5}. 
$$
Therefore
\begin{align*}
   h'(u)  & =- \frac{8 (u^3 -12 u   - 3 u \cos u + 15 \sin u)}{u^6}, \\
   h''(u)  & = \frac{24 (u ( u^2-20) - 10 u \cos u  - ( u^2 -30) \sin u)}{u^7}, \\
   h'''(u)  & = - \frac{24 (4 u ( u^2-30) + u ( u^2 -90) \cos u - 
    15 ( u^2-14) \sin u )}{u^8}. 
\end{align*}
It is immediately clear that  for all large $u$ the signs of the $h^{(j)}(u)$ alternate in such 
a way as to make $H_n(u) $ (defined in \eqref{Hn}) negative. We need a good
estimate of  just how large $u$ must be.

\bigskip
The signs $h^{(j)}(u)$ can be determined as in the previous cases.
It turns out that $h'(u) < 0$ for all $u > 0$. Elementary consideration shows that 
$h''(u) > 0$ for all large $u$ and numerical computation shows that $h''(u)$ has 
one simple zero for $u >0$ at $u_2 = 4.23573...$, so that $h''(u) > 0$ for $u > u_2$
and $h''(u) < 0$ for $0< u < u_2$. Similarly, $h'''(u)$ has one simple zero for $u > 0$ 
at $u_3 = 7.15125...$, $h_3'''(u) < 0$ for $u > u_3$  and $h'''(u) > 0$ for $0< u < u_3$.
We have several cases.

\medskip
{\it Case 1.} $u \geq u_3$. In this case, $h'(u) < 0$, $h''(u) > 0$ and $h'''(u) \leq 0$. That 
$H_n(u) < 0$ then follows immediately from the definition in \eqref{Hn}.

\medskip
{\it Case 2.} $0 < u< u_3$.  We write $H_n(u)$ as 
\begin{align*}
  H_n(u)  =  u^2 h'''(u) & - \frac{(3n+ 5 + \eta)(n+ \xi + \eta)}{(n+2)(n+ 1 + \eta)}
     \left[u h''(u) - \frac{n+ \xi + \eta}{n+ 1 + \eta}\, h'(u) \right]. 
\end{align*}
Since $h'(u) < 0$ it follows readily that for $0 < \xi, \eta  < 2 $, 
\begin{align*}
& u h''(u) - \frac{n+ \xi + \eta}{n+ 1 + \eta}\, h'(u)   
   > u h''(u) - \left(1-\frac{1}{n+1+\eta}\right)  h'(u) \\
     & \quad > u h''(u) - \left(1-\frac{1}{n+1}\right)  h'(u) \ge uh''(u) - h'(u)/2 > 0,
\end{align*}
if $u > u_0 = 2.99521 ...$, and the last quantity on the right of the display is zero
at $u=u_0$. Consequently, we obtain that for $u > u_0$, 
$$
H_n(u) < u^2 h'''(u) -  \left(3-\frac{1}{n+2}\right) \left(1-\frac{1}{n+1}\right)\left[u h''(u) - \left(1-\frac{1}{n+1}\right)  h'(u) \right].
$$
Denote the right hand side of the above inequality by $\Lambda_n(u)$. $\Lambda_n(u)$ is positive at
$u_0$ if $n=1$ and is negative at $u_3$ for all $n\in \N$. It is also a decreasing function of $n$ for
$u_0 < u $. Numerical computation shows that
$$
    \Lambda_{9}(u) < 0, \qquad \hbox{if} \quad u > u^* = 3.63661 ... . 
$$ 
This shows that $H_n(u) < 0$ for $ u^* < u \leq u_3$  and $n \ge 9$.
 
\medskip
We have already shown in case~1
that $H_n(u) <0$ if $u \geq u_3$. Hence $H_n(u) < 0$  on $(u^*,\infty)$, if $n \geq 9$.
As $u = (n+\xi+\eta)t$,
it follows by \eqref{Gn23} that $\displaystyle \frac{d}{dn} G_n^{2,3}(t) > 0$ if $n t > u^*$ or $t > u^*/n$.
Consequently, by \eqref{inductG} and \eqref{Gnlambda} we conclude that $ F^3_n(t)= 
C^3_n(1)\,G_n^{3,3}(t)  > 0$ if $t > u^* /n$ and $n \ge 9$.

On the other hand, an inequality in \cite{Ar04} shows that the largest zero of the Gegenbauer polynomial $C_n^\l$ satisfies the inequality
$$                                                                              
    x_{n,1}(\l) \le \sqrt{\frac{(n-1)(n+2\l -2)}{(n+\l-2)(n+\l -1)}} \cos \frac{\pi}{n+1}.                                                                    
$$
Therefore, defining $t_n^* =\arccos( x_{n,1}(3))$, 
$$
\cos^2 t_n^* \leq \frac{(n-1)(n+4)}{(n+1)(n+2)}\cos^2 \frac{\pi}{n+1},
$$
 and
\begin{equation} \label{eq_bound_lam3}
 \sin t_n^* \ge \sqrt{ 1 -\left( 1 -\frac{6}{(n+1)(n+2)} \right) 
\cos^2 \frac{\pi}{n+1}}.
\end{equation}
$C^3_n(\cos \theta)$ is positive for $0 <  \theta < t_n^*$, and therefore, from its definition,
$F^3_n(t)$ is positive 
for $0 < t \leq \sin( t_n^*)$.
Estimating $\sin t^*_n$ from below we have
\begin{equation}
\sin t_n^*  \geq \sqrt{ 1 - \left( 1 - \frac{6}{(n+1)^2}\right) \left( 1 -\frac{\pi^2}{2(n+1)^2} 
+ \frac{\pi^4}{24 (n+1)^4 } \right)^2 }   \\
= \frac{\sqrt{6+\pi^2}}{n+1} + {\mathcal O} \left( \frac{1}{n^2} \right). \label{sin_t_est_lam_3}
\end{equation}
At this point we have shown
$F^3_n(t)$ to be positive on $(0, \sin(t_n^*)]$ and also on $(u^*/n, \pi]$.
 Since $\sqrt{6+\pi^2} = 3.983667.. > u^*$ the asymptotic estimate of $\sin(t_n^*)$ above shows
that the regions on which $F_n^3(t)$ is positive overlap, and cover all of $(0,\pi]$,
for all large enough $n$. Numerical comparison of $\sin t_n^* $ and $u^*/n$ shows that the overlap
happens for all $n> 14$.  The proof of the positivity of $F^3_n(t)$ on $(0,\pi]$,  when $0\leq n \leq 14 $, can be completed by plotting the explicit expression
 (\ref{eq:F3}) for $F^3_n(t)$  on $[0,\pi]$. \qed
\newpage

\renewcommand{\section}[2]{\medskip{\large\bf #2}\stepcounter{section}}

\bigskip

\begin{tabbing}
xxxxxxxxxxxxxxxxxxxxxxxxxxxxxxxxxxxxxx\=xxxxxxxxxxxxxxxxxxxxxxxxxxxxxxx\kill
R. K. Beatson\> Wolfgang zu Castell\\
Department of Mathematics and Statistics\> Department of Scientific Computing\\
University of Canterbury\> Helmholtz Zentrum M\"{u}nchen\\
Private Bag 4800\>German Research Center  for\\
Christchurch, New Zealand \> \hspace{06ex}  Environmental Health\\
{\tt r.beatson@math.canterbury.ac.nz}\> Ingolst\"{a}dter Landstra\ss\,e 1\\
 \> 85764 Neuherberg, Germany\\
\> {\tt castell@helmholtz-muenchen.de}\\
Yuan Xu\\
Department of Mathematics\\
University of Oregon\\
Eugene, Oregon 97403-1222\\
U.S.A.\\
{\tt yuan@math.uoregon.edu}\\

\end{tabbing}


\begin{thebibliography}{99}

\bibitem{Abromowitz} Abromowitz, M. and I.A. Stegun (editors), {\it Handbook of Mathematical Functions},
   Dover Publications, Mineola, N.Y, 1965.

\bibitem{Ar04} Area, I., D. Dimitrov, E. Godoy \& A. Ronveaux,
Zeros of Gegenbauer and Hermite polynomials and connection
coefficients,
Math. Comp. {\bf 73} (2004), 1937--1951.

\bibitem{As73} Askey, R.,
Radial characteristic functions,
University of Wisconsin-Madison, Mathematics Research Center,
Tech. Report No. 1262, November 1973.

\bibitem{Askey_jmaa75} Askey, R., Some characteristic functions of unimodal distributions,
J. Math. Anal. Appl., {\bf 50} (1975), 465--469.
\bibitem{Ch03} Chen, D., Menegatto, V.A. \& Sun, X.,
A necessary and sufficient condition for strictly positive definite
functions on spheres,
Proc. Amer. Math. Soc. {\bf 131}, No. 9 (2003), 2733--2740.

\bibitem{El90} Elbert, A. \& A. Laforgia,
Upper bounds for zeros of ultraspherical polynomials,
J. Approx. Theory {\bf 61} (1990), 88--97.

\bibitem{Fi75} Fields, J.L. \& M.E. Ismail,
On the positivity of some $_1F_2$'s,
SIAM J. Math. Anal. {\bf 6}, No. 3 (1975), 551--559.

\bibitem{Ga75} Gasper, G.,
Positive integrals of Bessel functions,
SIAM J. Math. Anal. {\bf 6}, No. 5 (1975), 868--881.

\bibitem{Gn98} Gneiting, T.,
Simple tests for the validity of correlation function models
on the circle,
Stat. Probab. Lett. {\bf 39} (1998), 119--122.

\bibitem{Gn01} Gneiting, T.,
Criteria of P\'olya type for radial positive definite functions,
Proc. Amer. Math. Soc. {\bf 129}, No. 8 (2001), 2309--2318.

\bibitem{Po18} P\'olya, G.,
\"{U}ber die Nullstellen gewisser ganzer Funktionen (German),
Math. Zeitschr. {\bf 2} (1918), 352--383.

\bibitem{Po49} P\'olya, G.,
Remarks on characteristic functions,
Proc. Berkeley Symp. Math. Stat. Probab.,
University of California Press, 1949, pp. 115--123.

\bibitem{Royden} Royden, H.L.,
{\it Real Analysis}, 2nd edition,
Collier--Macmillan,
Toronto, Ontario, 1968.

\bibitem{Ru58} Rudin, W.,
Representation of functions by convolution,
J. Math. Mech. {\bf 7}, No. 1 (1958), 103--115.

\bibitem{Sc42} Schoenberg, I.J.,
Positive definite functions on spheres,
Duke Math. J. {\bf 9} (1942), 96--108.

\bibitem{Sz75} Szeg\"{o}, G.,
{\it Orthogonal Polynomials}, 4th Edition,
Amer. Math. Soc. Colloq. Publ., Vol. 23,
Providence, RI, 1975.

\bibitem{Tr89} Trigub, R.M.,
A criterion for a characteristic function and a P\'olya-type
criterion for radial functions of several variables,
Theor. Probab. Appl. {\bf 34} (1989), 738--742.

\bibitem{Xu92} Xu, Y. \& W. Cheney,
Strictly positive definite functions on spheres,
Proc. Amer. Math. Soc. {\bf 116} (1992), 977--981.

\end{thebibliography}
\end{document}